\documentclass[11pt]{amsart}

\usepackage{amsmath,amsthm,amssymb,amsfonts,amscd,
epsfig,latexsym,graphicx,epic,eucal,marvosym,textcomp,turnstile,extarrows}
\usepackage{hyperref}

\newtheorem{theorem}{Theorem}[section]
\newtheorem{lemma}[theorem]{Lemma}

\newtheorem{proposition}[theorem]{Proposition}

\theoremstyle{definition}

\theoremstyle{remark}
\newtheorem{remark}[theorem]{Remark}

\numberwithin{equation}{section}

\newcommand{\Tr}{\mathrm{Tr}}

\newcommand{\Ad}{\mathrm{Ad}}
\newcommand{\ad}{\mathrm{ad}}
\newcommand{\Z}{\mathbb{Z}}
\newcommand{\R}{\mathbb{R}}
\newcommand{\C}{\mathbb{C}}
\newcommand{\K}{\mathbb{K}}
\renewcommand{\b}[1]{\overline{#1}}

\newcommand{\eps}{\varepsilon}

\newcommand{\g}[1]{{\mathfrak{#1}}}
\newcommand{\im}{\mathrm{Im}}
\newcommand{\re}{\mathrm{Re}}

\newcommand{\wh}[1]{\widehat{#1}}

\newcommand{\SO}{\mathrm{SO}}
\newcommand{\GL}{\mathrm{GL}}
\newcommand{\SL}{\mathrm{SL}}
\newcommand{\cB}{\mathcal{B}}
\newcommand{\cL}{\mathcal{L}}
\newcommand{\cl}{\mathcal{C}^\lambda}
\newcommand{\ct}{\mathrm{Cos}^\lambda}
\newcommand{\st}{\mathrm{Sin}^\lambda}
\newcommand{\gp}{\mathrm{Gr}_p}

\def\sideremark#1{\ifvmode\leavevmode\fi\vadjust{\vbox to0pt{\vss
 \hbox to 0pt{\hskip\hsize\hskip1em
\vbox{\hsize2cm\tiny\raggedright\pretolerance10000 
 \noindent #1\hfill}\hss}\vbox to8pt{\vfil}\vss}}} 

\newcommand{\ipl}[2]{( #1,#2 )}
\newcommand{\ip}[2]{\langle #1,#2\rangle}

\begin{document}

\title[The $\ct$ transform]{The $\ct$ transform on line bundles over compact Hermitian symmetric spaces}

\author{Vivian M. Ho}
\address{Department of Mathematics \\
Louisiana State University\\
Baton Rouge\\\\ Louisiana}
\email{vivian@math.lsu.edu}

\author{Gestur \'{O}lafsson}
\begin{thanks}
{The research of G. \'Olafsson was supported by NSF grant DMS-1101337.} 
\end{thanks}

\address{Department of Mathematics \\
Louisiana State University\\
Baton Rouge\\\\ Louisiana}
\email{olafsson@math.lsu.edu}

\date{}

\begin{abstract}
In the article \cite{op} the authors determined the spectrum of the $\ct$ transform on 
smooth functions on the Grassmann manifolds $\gp (\K)$. This article extends
those results to line bundles over
certain Grassmannians. In particular we define  the $\ct$ transform on smooth sections of homogeneous 
line bundles over $\gp (\K)$ and show that it 
is an intertwining operator between generalized ($\chi$-spherical) principal series representations induced from
a maximal parabolic subgroup of $\SL (n+1, \K)$. Then we use the
spectrum generating method to determine the $K$-spectrum of the $\ct$ transform. 
\end{abstract}

\maketitle

\section{Introduction}
\label{introduction}
\noindent
The cosine transform is an integral
operator widely used in convex geometry.  
It enjoys a long and opulent history, with connection to several branches
of mathematics such as harmonic analysis, pseudo-differential
operators, convex geometry, and group representations.

The name cosine transform was first 
introduced by Lutwak \cite{lu} for the integral operator on the unit sphere 
\begin{equation}
\label{eq:cos1}
(\mathcal{C} f) (\omega) = \int_{S^n} |(x, w)| f (x)\, d x,\qquad \omega \in S^n
\end{equation}
where $d x$ is the normalized $\SO (n+1)$-invariant measure on $S^n$, $\ipl{\cdot}{\cdot}$ is the usual inner product on
$\R^n$ (so it is the cosine of the angle between two unit vectors). It was later extended to
a meromorphic family of integral operators  
\begin{equation}
\label{eq:cos2}
(\mathcal{C}^\lambda f) (\omega) = \int_{S^n} |(x, w)|^\lambda f (x)\, d x,\qquad \omega \in S^n
\end{equation}
where the integral (\ref{eq:cos2}) is understood, if needed, in the sense of analytic continuation.
We denote this transform by $\ct$.
There are distinct shifts in the power $\lambda$ in the literature based on distinct settings, see \cite{opr}. 
We use the shift $\lambda - \rho$ with $\rho = (n+1)/2$ so that (\ref{eq:cos2}) 
agrees with a standard intertwining operator between principal series representations of $\SL (n+1,\R)$.

Note that $\cl f = 0$ for all odd functions on $S^n$. 
So we can view the $\ct$ transform (\ref{eq:cos2}) as an integral transform on $L^2 (\mathrm{Gr}_1 (\R))$ 
where $\mathrm{Gr}_1 (\R)$ is the Grassmann manifold of one dimensional linear subspaces (lines) of $\R^{n+1}$. 
It was also extended to other Grassmann manifolds. See \cite{opr} and \cite{ru} for
more detailed discussion about the history of the cosine transform.

The operator (\ref{eq:cos2}) with the kernel $|(x, w)|^\lambda$ has been studied 
by many researchers in geometry and analysis like Aleksandrov \cite{al}, Gardner
\cite{gar}, Schneider \cite{sc}, and others, via different approaches, such as the Fourier transform technique, the
Funk-Hecke formula, or representation theory. The importance of the cosine transform comes from intimate connections
to convex geometry and classical integral transforms like the Fourier, Funk and Radon transforms. 
A remarkable fact, due to Gelfand-Shapiro \cite{gs} and Semyanistyi \cite{se}, is that
cosine transforms are restrictions to $S^n$ of the Fourier transforms of homogeneous distributions on $\R^n$. Another
important observation is the connection of the cosine transform to the Funk transform 
$$(\mathcal{F} f) (\omega) = \int\limits_{(x, \omega) = 0} f (x)\, d m (x), \quad \omega \in S^n$$
where $d m$ is the rotation-invariant measure on the $(n-1)$-dimensional sphere $(x, \omega) = 0$. If $f \in C (S^n)$, then
\[\lim\limits_{\lambda \to -1} \gamma_{n+1} (\lambda) (\cl f) (\omega) = \frac{\pi^{1/2}}{\Gamma (n/2)} (\mathcal{F} f) (\omega)\]
where the coefficient $\gamma_n (\lambda)$ is given by 
\[\gamma_n (\lambda) = \frac{\pi^{1/2} \Gamma (- \lambda / 2)}{\Gamma (n/2) \Gamma ((1 + \lambda)/2)},\quad 
\re \lambda > -1,\; \lambda \notin 2 \Z^+.\]

Several natural questions arise for the cosine transform, namely, to determine its kernel and image, to characterize its rank, 
to figure out the meromorphic extension of this operator and reveal singularities, and to derive the composition 
formula and the inversion formula. Many of 
the mentioned important information can be derived by knowing the spectrum of the cosine transform (it is formed by a set of 
spectral functions). The spectral functions of the operator (\ref{eq:cos2}) are well investigated. 
They are given, using a Funk-Hecke type formula, by rational functions of $\Gamma$-factors (Euler Gamma functions), 
see \cite{ru1} and also Section 2 in \cite{opr}.

In recent years, the cosine transform has been generalized to the more general context on the Stiefel and Grassmann
manifolds, for instance, Rubin \cite{ru} worked on more general higher-rank cosine and sine transforms 
where the main tool is the classical Fourier analysis. 
For some specific cases of higher-rank cosine transforms explicit forms for the spectral functions can still be obtained using 
the Funk-Hecke formula or the Fourier transform technique, see, for instance, \cite{or} and \cite{or1}. However, it is unknown and 
difficult to proceed the same way in the most general case.

The cosine transform is closely and naturally related to the representation theory of semisimple Lie groups. In particular, 
the integral transform (\ref{eq:cos2}), as mentioned earlier, has an important group-theoretic interpretation as a
standard intertwining operator $J (\lambda)$ between generalized principal series representations of $G = \SL (n + 1, \K)$
with $\K = \R$, $\C$, or $\mathbb{H}$ (the field of quaternions), in the compact picture. The real case was studied in \cite{dm}, the complex case in \cite{dz},
and the quaternionic case in \cite{pa}, without mentioning any connection to the cosine transform or convex geometry. To the best
of our knowledge, Alesker \cite{ale} was the first to remark that the cosine transform is a 
$\SL (n, \R)$-intertwining operator for some suitable $\lambda$.
The operators $J (\lambda)$, as a meromorphic family of singular integral operators on the
maximal compact subgroup $K$ of $G$, or on the orbit of certain nilpotent group $\overline{N}$, 
have been central objects in the study of representation theory of semisimple Lie groups.

Let $\gp (\K)$ be the Grassmann manifold of $p$-dimensional linear subspaces of $\K^{n + 1}$.
It was shown in \cite{op} that the $\ct$ transform for smooth functions on $\gp (\K)$ 
are always canonical intertwining operators
between generalized principal series representations induced from maximal parabolic subgroups of $G$. 
Using the spectrum generating method (developed in \cite{boo}), 
the $K$-spectrum of the $\ct$ transform was determined.

In this article, we consider the $\ct$ transform on smooth sections of homogeneous nontrivial line bundles over $\gp (\K)$. 
If $\K$ is the field $\mathbb{H}$ of quaternions, or the field $\mathbb{O}$ of octonions, then there are 
only trivial line bundles over $\gp (\K)$. Thus we restrict ourselves to the cases $\K = \R$ and $\C$, and 
whence the Grassmann manifold $\gp (\K)$ can be realized as a hermitian symmetric space $K/L$. 
Let $\chi$ be a nontrivial character of $L$. 
Let $L^2 (K/L; \cL_\chi)$ be the space of $L^2$-sections of line bundles $\cL_\chi$ on $K/L$.
We use the similar spectrum generating technique (build up a recursive relation between
the spectral functions $\eta_\mu$) 
to determine the $K$-spectrum $\eta_\mu (\lambda)$ of the
$\ct$ transform for smooth sections of line bundles over $K/L$. 
As in \cite{op} we use the spectrum generating operator to calculate the spectrum. For that it is
needed that the $\chi$-spherical representations have multiplicity one in $L^2(K/L;\cL_\chi)$. We also
need a good description of the set of highest weights of $\chi$-spherical representations. The calculation of
the $K$-spectrum is then reduced to evaluation of a Siegel type integral for the ``smallest'' $K$-type,
see Theorem \ref{thm3}.

The article is organized as follows. In Section \ref{parasg} we introduce notations and facts on simple Lie groups and Lie
algebras. In Section \ref{psr} we introduce the generalized 
($\chi$-spherical) principal series representations of $G$. They are representations induced
from a character of a parabolic subgroup $P = M A N$. The elements of the representation space can be viewed as $L^2$
sections of the $G$-homogeneous line bundles over $\cB = G/P$. Then we define the intertwining operator $J (\lambda)$ between
generalized ($\chi$-spherical) principal series representations. In Section \ref{sgo} we specialize to the case $P$ is a 
maximal parabolic subgroup of $G$ and the $K$-types has multiplicity one. Then we apply the spectrum generating method and 
give a recursive relation (\ref{eq10}) between the eigenvalues $\eta_\mu (\lambda)$ of $J (\lambda)$ on each of the
$K$-types. Section \ref{grassmann} is devoted to the cases of the hermitian Grassmann manifolds $\gp (\K) = G/P = K/L$. 
We detail the 
classification of $\chi$-spherical representations of $K$, and prove some useful properties of $\chi$-spherical functions for 
these cases. In Section \ref{spectrum}, we stick to the cases considered in Section \ref{grassmann}, and the recursive formula 
(\ref{eq10}) boils down to an explicit form (\ref{recur}).
We introduce the $\ct$ transform $\mathcal{C}^\lambda$ on homogeneous line bundles over $K/L$ and 
conclude that $\mathcal{C}^\lambda = J (\lambda)$ which allows us to
compute the $K$-spectrum of $\mathcal{C}^\lambda$. An initial eigenvalue $\eta_{\mu^0} (\lambda)$ (the smallest one)
is computed in (\ref{eq19}) and all the others then follow by an inductive procedure using (\ref{recur}). Finally, (\ref{eq20}) is derived 
as the formula of the eigenvalues $\eta_\mu (\lambda)$. This is the main result of this article.

\section{Semisimple Lie groups and parabolic subgroups}
\label{parasg}

\noindent
Let us settle on some definitions and facts related to semisimple Lie groups, Lie algebras and parabolic subgroups. 
The material in this section is standard.
A good reference is, e.g. \cite{kn}. We adopt most of the notations in \cite{op}.

Let $G$ be a noncompact connected semisimple Lie group with finite center and $\g{g}$ the Lie algebra of $G$. Let $\theta$ be
a fixed Cartan involution on $G$ and $K = G^\theta$ the corresponding maximal compact subgroup of $G$. 
Note that $K$ is connected. The derived involution on $\g{g}$ and its complex linear extension to 
$\g{g}_{\C} = \g{g} \otimes_\R \C$ are still denoted by $\theta$. Then
$\g{g} = \g{k} \oplus \g{s}$ where $\g{k} = \g{g}^\theta$ is the Lie algebra of $K$ and 
$\g{s} = \{X \in \g{g}\, \mid\, \theta (X) = - X\}$. 
Denote by $\ip{ \cdot}{\cdot}$ the Cartan-Killing form on $\g{g}$.
Let $\g{a} \subset \g{s}$ be abelian. Let 
$$\g{m} = \{X \in \g{g}\; \mid\; [H, X] = 0\; \text{and}\; X \perp \g{a},\; \forall H \in \g{a}\}.$$ 
We assume that $\g{a} = \g{z} (\g{m}) \cap \g{s}$. Let $\Delta = \Delta (\g{g}, \g{a})$ be the set of roots of $\g{g}$ with
respect with $\g{a}$. For $\alpha \in \Delta$, denote by $\g{g}_\alpha$ the root space. Fix a choice of system $\Delta^+$ of
positive roots. Let
$$\g{n} := \bigoplus_{\alpha \in \Delta^+} \g{g}_\alpha$$
and we have $\g{p} := \g{m} \oplus \g{a} \oplus \g{n}$ is a parabolic subalgebra of $\g{g}$. It is maximal if $\dim \g{a} = 1$.
Let 
\[P= \{g \in G\; \mid\; \Ad (g) \g{p} \subset \g{p}\}.\]
Then $P$ is a closed subgroup of $G$ with Lie algebra $\g{p}$.
Let $A = \exp \g{a}$ and $N = \exp \g{n}$ denote the analytic subgroups of $G$ with Lie algebras $\g{a}$ and $\g{n}$,
respectively. Let $M_0$ be the analytic subgroup of $G$ with Lie algebra $\g{m}$ and $M := Z_K (A) M_0$. Then $Z_G (A) = M A$. 
The map 
\[M \times A \times N \longrightarrow P,\qquad (m, a, n) \longmapsto m a n\]
is an analytic diffeomorphism. 
Let $L = K \cap M$ and $\cB := K / L$. Then $G = K P$ and $K \cap P = L$. Thus $\cB = G / P$. For $g \in G$, write $g =
\kappa_P (g) m_P (g) a_P (g) n_P (g)$ where
$$(\kappa_P (g), m_P (g), a_P (g), n_P (g)) \in K \times M \times A \times N.$$
The subscript $_P$ (here and elsewhere) is omitted if no confusion occurs. The elements $\kappa (g)$ and $m (g)$ are not
uniquely determined. The map $g \mapsto a (g)$ is right $M N$-invariant.
Let $b_0 = e L$. We define the action of $G$ on $\cB$ by 
\[g \cdot (k \cdot b_0) = \kappa (g k) \cdot b_0.\]
Let $\b{N} := \theta (N)$ and $\b{P} = M A \b{N} = \theta (P)$. The Lie algebra of $\b{N}$ is $\b{\g{n}} = \theta (\g{n})$.
We have 
\[\g{g} = \b{\g{n}} \oplus \g{m} \oplus \g{a} \oplus \g{n} = \b{\g{n}} \oplus \g{p}.\]
For $g \in \b{N} P$ write 
\[g = \b{n}_P (g) m_P (g) \alpha_P (g) n_P (g)\]
where $\b{n}_P (g) \in \b{N}$, $m_P (g) \in M$, $\alpha_P (g) \in A$, $n_P (g) \in N$. 
Note that $g \mapsto \alpha_P (g)$ is right $M N$-invariant. The map 
\[\b{N} \times M \times A \times N \longrightarrow G,\qquad (\b{n}, m, a, n) \longmapsto \b{n} m a n\]
is an analytic diffeomorphism onto an open dense subset $\b{N} P$ of $G$. 
Note that as a $K$-manifold $\cB = G / \b{P}$ but $G$ acts on it distinctly. Denote this action by $\b{\cdot}$.

\begin{lemma}
Let $k \in K$ and $h_1, h_2 \in L$. Then
\begin{equation}
\label{eq22}
m (h_1 k h_2) = h_1 m (k) h_2,\qquad \alpha (h_1 k h_2) = \alpha (k).
\end{equation}
\end{lemma}

\begin{proof}
We have for $k \in K$ and $h \in L$,
\[h k = h (\overline{n} (k) m (k) \alpha (k) n (k)) = (h \overline{n} (k) h^{-1}) (h m (k)) \alpha (k) n (k)\, .\]
Since $L$ normalizes $\overline{N}$, $h \overline{n} (k) h^{-1} \in \b{N}$. 
Thus, 
\[h m (k) = m (h k)\qquad \text{and}\qquad\alpha (k) = \alpha (h k).\]
Similarly, 
\[k h = \b{n} (k) (m (k) h) (h^{-1} \alpha (k) h) (h^{-1} n (k) h).\]
Using the facts that $M$ normalizes $N$, and $M$ centralizes $A$, we get 
\[m (k) h = m (k h)\qquad \text{and} \qquad \alpha (k) = \alpha (k h).\] 
\end{proof}

For $\alpha \in \Delta$, let $m_\alpha = \dim_\R \g{g}_\alpha$ and define 
\[\rho = \rho_P = \frac{1}{2} \sum_{\alpha \in \Delta^+} m_\alpha \alpha \in \g{a}^\ast.\]
We normalize the invariant measure on $\cB$ and compact groups so that the total measure is one. 
Normalize the Haar measure $d \b{n}$ on $\b{N}$ by
\begin{equation}
\label{mea:nbar}
\int_{\b{N}} a (\b{n})^{- 2 \rho} d \b{n} = 1.
\end{equation}
The following are some integral formulas we need for our proof below. 
They can be found in any standard reference on symmetric spaces, e.g. \cite{kn}.

\begin{lemma}
Let $f \in L^1 (\cB; \cL_\chi)$ and $g \in G$. We have
\begin{equation}
\label{eq5}
\int_K f (\kappa (g^{-1} k)) a (g^{-1} k)^{- 2 \rho} d k = \int_K f (k) d k.
\end{equation}
Moreover, we can transfer an integral over $\b{N}$ to an integral over $K$: 
\begin{equation}
\label{eq2}
\int_K f (k) d k = \int_{\overline{N}} f (\kappa (\overline{n})) a (\overline{n})^{-2 \rho} d \overline{n}.
\end{equation}
\end{lemma}

\section{Generalized principal series representations}
\label{psr}

\noindent
In this section, we introduce the generalized $\chi$-spherical principal series representations and their intertwining operators.
We refer to \cite{ks, ks1, vw} for more information about this subject. Some of theorems in this section quote from 
the mentioned resource without giving proofs.
We will show in Section \ref{spectrum} that these intertwining operators coincide with the $\ct$ transform on smooth sections 
of homogeneous line bundles over $K/L$.

Let $\chi$ be a (nontrivial) character of $L$.
Let $[L, L]$ be the commutator subgroup of $L$. Since $[M, M] \cap L = [L, L]$
and $\chi |_{[L, L]} = 1$, we can extend $\chi$ to a character on $M$ (still called $\chi$) by defining 
\begin{equation}
\label{eq:chi-M}
\chi (m) = \chi (a \exp (X)) = \chi (a),\qquad a \in L,\; X \in \g{m} \cap \g{s}.
\end{equation}
Note that $\chi$ is unitary. Recall that $\theta$ is a fixed Cartan involution on $G$.

\begin{lemma}
\label{lem:comp}
We have $\chi \circ \theta = \chi$, where $\theta = \theta |_M$.
\end{lemma}

\begin{proof}
Since $M = L \exp (\g{m} \cap \g{s})$, we can write any $m \in M$ as 
\[m = k \exp (X),\quad k \in L, X \in \g{m} \cap \g{s}.\]
Then we have 
\[\theta (m) = \theta (k) \theta (\exp X) = k \exp (-X).\]
The character $\chi$ is trivial on $\g{s}$. 
Therefore, $\chi |_{\exp (\g{m} \cap \g{s})} = \mathrm{Id}$. Hence,
\[\chi (\theta (m)) = \chi (k \exp (-X)) = \chi (k) = \chi (k \exp (X)) = \chi (m).\]
This implies that $\chi \circ \theta = \chi$.
\end{proof}

For a nontrivial character $\chi$ of $L$, the homogeneous line bundle $\cL_{\chi} \to \cB$ is defined as 
\[\cL_\chi = K \times_{\chi} \C = (K \times \C) / \sim\]
where $\sim$ is the equivalence relation given by $(k m, \chi (m)^{-1} z) \sim (k, z)$ for $k \in K$, $m \in L$, and $z \in \C$.
Denote the space 
$$C^\infty (\cB; \cL_\chi) = \{f \in C^\infty (K)\; \mid\; f (k m) = \chi (m)^{-1} f (k), \forall m \in
L\}.$$  
We may think of these functions geometrically as smooth sections in $\cL_\chi$. 
Similarly, $L^2 (\cB; \cL_\chi)$ consists of square integrable sections. 
Take a character $\lambda$ of $A$ ($\lambda \in \g{a}^\ast$), we form a character of $P$ by
$\chi \otimes \lambda \otimes 1$. For $\lambda \in \g{a}_{\C}^\ast$ we define a continuous representation 
of $G$ on $L^2 (\cB; \cL_\chi)$ by 
\begin{equation}
\label{eq1}
(\pi_\lambda (g) f) (k) = \chi (m (g^{-1} k))^{-1} a (g^{-1} k)^{- \lambda - \rho} f (g^{-1} \cdot k).
\end{equation}
So $\pi_\lambda = \pi (P, \chi, \lambda) = \mathrm{Ind}_P^G \chi \otimes \lambda \otimes 1$. 
This is the generalized
($\chi$-spherical) principal series representation of $G$ induced from the character $\chi \otimes \lambda \otimes 1$ 
of a parabolic subgroup $P$ of $G$.

\begin{theorem}
The representation $\pi_\lambda$ is irreducible for almost all $\lambda \in \g{a}_\C^\ast$.
\end{theorem}

Note that $\pi_\lambda$ is unitary if and only if $\lambda \in i \g{a}^\ast$. We can realize $\cB$ as $G / \b{P}$ and 
induce the corresponding representation $\pi (\b{P}, \chi, \lambda)$.

\begin{lemma}
Let $\pi_\lambda^\theta = \pi_\lambda \circ \theta$. Then $\pi_{-\lambda}^\theta = \pi (\b{P}, \chi, \lambda)$ as
a representation of $G$ on $L^2 (\cB; \cL_\chi)$. 
\end{lemma}

\begin{proof}
It follows from Lemma 1.1 in \cite{op} that for $g \in G$ and $k \in K$
\[m_{\b{P}} (g k) = \theta [m_P (\theta (g) k)],\quad a_{\b{P}} (g k) = a_P (\theta (g) k)^{-1},\quad g \b{\cdot} k = \theta (g) \cdot k\, .\]
Using the facts that $M$ is $\theta$-stable, and $\rho_{\b{P}} = - \rho_P$, together with Lemma \ref{lem:comp}, we get
\begin{align*}
(\pi_{\b{P}, \chi, \lambda} (g) f) (k) & =  \chi (m_{\b{P}} (g^{-1} k))^{-1} a_{\b{P}} (g^{-1} k)^{-
\lambda - \rho_{\b{P}}} f (g^{-1} \b{\cdot} k)\\
& =  \chi (\theta (m_P (\theta (g^{-1}) k))^{-1}) a_P (\theta (g^{-1}) k)^{\lambda - \rho_P} f
(\theta(g^{-1}) \cdot k)\\
& =  \chi (m_P (\theta (g)^{-1} k))^{-1} a_P (\theta (g)^{-1} k)^{\lambda - \rho_P} f (\theta(g)^{-1} \cdot
k)\\
& =  (\pi_{P, \chi, - \lambda} (\theta (g)) f) (k).
\end{align*}
\end{proof}

For $f \in C^\infty (\cB; \cL_\chi)$ we define an intertwining operator $J (\lambda) = J_\chi (\lambda)$ on $C^\infty (\cB; \cL_\chi)$ by 
\begin{equation}
\label{eq:J}
J (\lambda) f (k) = \int_{\b{N}} a (\b{n})^{- \lambda - \rho} \chi (m (\b{n}))^{-1} f (k\; \kappa
(\b{n}))\, d \b{n}
\end{equation}
whenever the integral exists. Note that this integral is well defined.

\begin{proposition}
There is a constant $c_P$ such that by defining
\[\g{a}_{\C}^\ast (c_P) := \{\lambda \in \g{a}_{\C}^\ast\; \mid\; (\forall \alpha \in \Delta^+)\;
\langle \re \lambda,\; \alpha \rangle \geq c_P\}\]
the integral (\ref{eq:J}) converges for $\lambda \in \g{a}_{\C}^\ast (c_P)$.
\end{proposition}

An explicit choice of $c_P$ is given as follows:
\[c_P = \max_{\alpha \in \Delta^+} \langle \alpha, \rho\rangle\]
For $\lambda \in \g{a}_{\C}^\ast (c_P)$, $|a (\overline{n})^{- \lambda - \rho}| \leq a (\overline{n})^{- 2 \rho}$.
The map $\b{n} \mapsto a (\b{n})^{-2 \rho}$ is integrable. Also, $|\chi (m (\b{n}))^{-1}| = 1$.
Therefore, the integral (\ref{eq:J}) converges for $\lambda \in \g{a}_{\C}^\ast (c_P)$ and 
$f \in C^\infty (\cB; \cL_\chi)$. Moreover, there is a constant $C>0$ such that
\begin{equation}
\label{eq:infty}
\|J (\lambda) f\|_\infty \leq C \|f\|_\infty.
\end{equation}
The map $\lambda \mapsto J (\lambda) f$ extends to a meromorphic function on $\g{a}_\C^\ast$.

\begin{lemma}
Let $\lambda \in \g{a}_\C^\ast (c_P)$. Let $1 \leq p \leq \infty$ and $f \in L^p (\cB; \cL_\chi)$. Then the integral (\ref{eq:J})
exists and $|J (\lambda) f (k)| \leq \|f\|_p$. So $J (\lambda) f \in L^p (\cB; \cL_\chi)$ and 
$\|J (\lambda) f (k)\| \leq \|f\|_p$.
\end{lemma}

\begin{proof}
We get the case $p = \infty$ in (\ref{eq:infty}). Let $1 \leq p < \infty$. We have 
\[|J (\lambda) f (k)| \leq \int_{\overline{N}} a (\overline{n})^{- 2 \rho} |f (k \kappa (\overline{n}))| d \overline{n}\]
where we use the facts that $|a (\overline{n})^{- \lambda - \rho}| \leq a (\overline{n})^{- 2 \rho}$ 
and $|\chi (m (\b{n}))^{-1}| = 1$. The conclusion $|J (\lambda) f (k)| \leq \|f\|_p$
follows from H\"{o}lder Inequality, (\ref{mea:nbar}), and (\ref{eq2}).
\end{proof}

\begin{lemma}
Let $\lambda \in \g{a}_{\C}^\ast (c_P)$ and $f \in C^\infty (\cB; \cL_\chi)$. Then for $k \in K$ 
\begin{equation}
\label{eq3}
J (\lambda) f (k) = \int_K \alpha (h)^{\lambda - \rho} \chi (m (h)) f (k h)\, d h.
\end{equation}
\end{lemma}

\begin{proof}
We use the facts that $m(\kappa(\overline{n})) = m (\overline{n})^{-1}$ and $\alpha (\kappa (\overline{n})) = a
(\overline{n})^{-1}$ (see the proof of Lemma 2.5 in \cite{op}), and the formula (\ref{eq2})
to rewrite $J (\lambda)$ as an integral over $K$:
\begin{align*}
J (\lambda) f (k) & =  \int_{\b{N}} a (\b{n})^{-\lambda + \rho} \chi (m (\b{n}))^{-1} 
f (k \kappa (\b{n})) a (\b{n})^{-2 \rho}\, d \b{n}\\
& =  \int_{\overline{N}} \alpha (\kappa (\overline{n}))^{\lambda - \rho} \chi (m (\kappa
(\overline{n}))) f (k \cdot \kappa (\overline{n})) a (\overline{n})^{- 2 \rho}\, d \overline{n}\\
& =  \int_K \alpha (h)^{\lambda - \rho} \chi (m (h)) f (k h)\, d h.
\end{align*} 
\end{proof}

Moreover, by changing the variable $h \mapsto k^{-1} h$ we can rewrite (\ref{eq3}) as
\[J (\lambda) f (k) = \int_K \alpha (k^{-1} h)^{\lambda - \rho} \chi (m (k^{-1} h)) f (h)\, d h.\]
We now show $J (\lambda)$ is an intertwining operator.

\begin{theorem}
\label{thm1}
For $\lambda \in \g{a}_{\C}^\ast$, the operator $J (\lambda)$ intertwines $\pi_\lambda$ and $\pi_{-\lambda}^\theta$. 
That is, for $g \in G$ and $f \in C^\infty (\cB; \cL_\chi)$,
\begin{equation}
\label{eq6}
J (\lambda) (\pi_\lambda (g) f) = \pi_{- \lambda} (\theta (g)) J (\lambda) f.
\end{equation}
\end{theorem}

\begin{proof}
Let $g \in G$ and $\lambda \in \g{a}_{\C}^\ast (c_P)$. Let $k \in K$. Using the fact that $g \mapsto \alpha (g)$ is 
right-$M N$ invariant, we have 
\begin{equation}
\label{eq25}
\alpha (k^{-1} h) = \alpha (k^{-1} g g^{-1} h) = \alpha (k^{-1} g \kappa (g^{-1} h)) a (g^{-1}
h),\quad h \in K.
\end{equation}
It follows from (\ref{eq1}), (\ref{eq3}), and (\ref{eq25}) that
\begin{align*}
& J (\lambda) (\pi_\lambda (g) f) (k) \\
& =  \int_K \alpha (k^{-1} h)^{\lambda - \rho} 
\chi (m (k^{-1} h)) a (g^{-1} h)^{- \lambda - \rho} \chi (m (g^{-1} h))^{-1} f (\kappa (g^{-1} h))\, d h \\
& =  \int_K \alpha (k^{-1} g \kappa (g^{-1} h))^{\lambda - \rho} a (g^{-1} h)^{- 2 \rho} \\
& \cdot \chi (m (k^{-1} h)) \chi (m (g^{-1} h))^{-1} f (\kappa (g^{-1} h))\, d h\,. 
\end{align*}
Note that 
\begin{eqnarray*}
k^{-1} h & = & k^{-1} g g^{-1} h\, =\, k^{-1} g \kappa (g^{-1} h) m (g^{-1} h) a (g^{-1} h) n
(g^{-1} h),\\
m (k^{-1} h) & = & m (k^{-1} g \kappa (g^{-1} h)) m (g^{-1} h).
\end{eqnarray*}
Therefore,
\begin{equation}
\label{eq27}
\chi (m (k^{-1} h)) \chi (m (g^{-1} h))^{-1} = \chi (m (k^{-1} g \kappa (g^{-1} h)).
\end{equation}
Using (\ref{eq5}) and (\ref{eq27}), we have
\begin{eqnarray*}
& & J (\lambda) (\pi_\lambda (g) f) (k) \\
& = & \int_K \alpha (k^{-1} g \kappa (g^{-1} h))^{\lambda - \rho} a (g^{-1} h)^{- 2 \rho}
\chi (m (k^{-1} g \kappa (g^{-1} h)) f (\kappa (g^{-1} h))\, d h\\
& = & \int_K \alpha (k^{-1} g h)^{\lambda - \rho} \chi (m (k^{-1} g h)) f (h)\, d h\\
& = & \int_K \alpha ((g^{-1} k)^{-1} h)^{\lambda - \rho} \chi (m ((g^{-1} k)^{-1} h)) f (h)\, d h.
\end{eqnarray*}
Since 
$$g^{-1} k = \kappa_P (\theta (g^{-1}) k) \theta (m_P (\theta (g^{-1}) k)) a_P (\theta (g^{-1})
k)^{-1} \theta (n_P (\theta (g^{-1}) k)),$$
then
\begin{align*} (g^{-1} &k)^{-1} h\\
& =  \underbrace{\theta (n (\theta (g^{-1}) k))^{-1}}_{\in \overline{N}}
\underbrace{a (\theta (g^{-1}) k)}_{\in A} \underbrace{\theta (m (\theta (g^{-1})
k))^{-1}}_{\in M} \underbrace{\kappa (\theta (g^{-1}) k)^{-1} h}_{\in K}.
\end{align*}
Thus, we have
\begin{eqnarray*}
\alpha ((g^{-1} k)^{-1} h) & = & a (\theta (g^{-1}) k)\; \alpha (\kappa (\theta (g^{-1}) k)^{-1} h),\\
m ((g^{-1} k)^{-1} h) & = & \theta (m (\theta (g^{-1}) k))^{-1} m (\kappa (\theta (g^{-1}) k)^{-1} h),\\
\chi (m ((g^{-1} k)^{-1} h)) & = & \chi  (m (\theta (g^{-1}) k)^{-1}) \chi (m (\kappa (\theta
(g^{-1}) k)^{-1} h)),
\end{eqnarray*}
where we use the fact that $\chi \circ \theta = \chi$ (see Lemma \ref{lem:comp}). Hence,
\begin{align*} 
J (\lambda) (\pi_\lambda (g) f) (k) 
& =  a (\theta (g)^{-1} k)^{\lambda - \rho} \chi (m (\theta (g)^{-1} k))^{-1}\cdot\\
& \quad \cdot  \int_K \alpha (\kappa
(\theta (g)^{-1} k)^{-1} h)^{\lambda - \rho} \chi (m (\kappa (\theta (g)^{-1} k)^{-1} h)) f (h)\, d h\\
& =  \pi_{- \lambda} (\theta (g)) J (\lambda) f (k).
\end{align*}
So (\ref{eq6}) for $\lambda \in \g{a}_{\C}^\ast (c_P)$ follows. It holds for all $\lambda \in \g{a}_{\C}^\ast$ by meromorphic continuation. 
\end{proof}

\section{The spectrum generating operator}
\label{sgo}

\noindent
In what follows we shall discuss the connections to the $\chi$-spherical representations and the intertwining properties. 
We will explain how to apply the spectrum generating method (proposed by \cite{boo}) to our cases to achieve the eigenvalues 
of $J (\lambda)$ on each of the $K$-types. The chief result is the recursion relation (\ref{eq10}).

{}From now on we specialize the setting to the following case. Assume $K$ is semisimple. Let $K/L$ be a compact symmetric space such
that there exists a noncompact connected semisimple Lie group $G$ with $G/P = K/L = \cB$ where $P = M A N$ is a maximal
parabolic subgroup of $G$ and $N$ is abelian. So $\dim \g{a} = 1$. 
Furthermore, we assume $K/L$ is of hermitian type so that there are nontrivial one dimensional representations
$\chi$ of $L$ occurring. The symmetric pair $(K, L)$ corresponds to an involution $\tau$ on $K$ so that $K^\tau = L$. Each
irreducible representation of $K$ that occurs in $L^2 (\cB; \cL_\chi)$ has multiplicity one.

Let $\g{k} = \g{l} \oplus \g{q}$ be the Cartan decomposition where $\g{l}$ is the Lie algebra of $L$ and 
\[\g{q} = \{X \in \g{k} \mid \tau (X) = - X\}.\]
Let $\g{b}$ be a maximal abelian subspace of $\g{q}$ and $t = \dim \g{b}$. 
Denote by $\Sigma$ the set of nonzero restricted roots of the pair $(\g{k}, \g{b})$. Since all
the elements of $\Sigma$ are purely imaginary on $\g{b}$, we see that $\Sigma \subset i \g{b}^\ast$. 
Let $W$ be the Weyl group associated with $\Sigma$. A multiplicity function $m:\Sigma \to \C$ is a $W$-invariant function. 
We write $m_\alpha = m (\alpha)$ for $\alpha \in \Sigma$.
Fix a set $\Sigma^+ \subset \Sigma$ of positive roots. 
Let $\rho_{\g{k}} = \rho_{\g{k}} (m)$ be given by
\begin{equation}
\label{eq:rhok}
\rho_\g{k} = \frac{1}{2} \sum_{\alpha \in \Sigma^+} m_\alpha \alpha.
\end{equation}
We have the Cartan decomposition $K=L(\exp \g{b})L$, and $(\exp \g{b}) \cap L$ can be bigger than $\{e\}$.

\begin{lemma}  
For $f \in L^1 (\cB; \cL_\chi)$, there exists a constant $C$ such that  
\begin{equation}
\label{eq18}
\int_K f (k) d k = C \int_L \int_{\exp \g{b}} \int_L f (x b y) \delta (b)\, d y\, d b\, d x
\end{equation}
where $\delta = \prod_{\alpha \in \Sigma^+} |e^\alpha - e^{- \alpha}|^{m_\alpha}$.
\end{lemma}

Let $(\pi, V)$ be an irreducible unitary representation of $K$ and $\chi$ a nontrivial character of $L$. Let 
$$V^\chi := \{v \in V\; \mid\; \pi (m) v = \chi (m) v,\; \forall m \in L\}.$$
The representation $(\pi, V)$ is said to be \textit{$\chi$-spherical} if $V^\chi \ne \{0\}$. In this case, $\dim V^\chi = 1$. 
Let $\wh{K}_\chi$ be the set of equivalence classes of irreducible $\chi$-spherical representations of $K$. 
Let $\Lambda_\chi^+ (K)$ 
denote the set of highest weights of irreducible $\chi$-spherical representations of $K$. Let  
\begin{equation}\label{eq:Lplus}
\Lambda_\chi^+ := \{\mu \in i \g{b}^\ast\; \mid\; \mu = \lambda |_\g{b},\; \lambda \in \Lambda_\chi^+ (K)\} .
\end{equation}
Then $\lambda$ has the form $\lambda |_{\g{b}} + \mu_0$ where $\mu_0 \in i (\g{b}^\perp)^\ast$ is independent of $\lambda$
(see Section 3 of \cite{hol} for more details). 
Thus there is a bijective correspondence $\Lambda_\chi^+ (K) \cong \Lambda_\chi^+$. 
Then $\pi \mapsto \mu_\pi$ ($\mu_\pi$ is the highest weight of $\pi$) gives an injective map from $\wh{K}_\chi$ into 
$\Lambda_\chi^+$. It is bijective if and only if $\cB$ is simply connected. For non-simply-connected cases,
$\wh{K}_\chi$ is isomorphic to a sublattice of $\Lambda_\chi^+$, because there are additional conditions need to 
be derived for the highest weights so that they are one-to-one correspondence with irreducible representations.

Next, we discuss how to find the eigenvalues of the intertwining operator $J (\lambda)$ on
each of the $K$-types. The idea is inspired by \cite{op}.

Let $\mu \in \Lambda_\chi^+$ and denote by $(\pi_\mu, V_\mu)$ the corresponding $\chi$-spherical
unitary representation of $K$. Put $d (\mu) = \dim V_\mu$. Let $\ipl{\cdot}{\cdot}$ be the inner product in $V_\mu$ for
which $\pi_\mu$ is unitary. Fix a unit vector $e_\mu \in V_\mu^\chi$. Define $T_\mu: V_\mu \to L^2 (\cB; \cL_\chi)$ by 
$$T_\mu (v) (\; \cdot\; ) := d (\mu)^{1/2} (v,\; \pi_\mu (\, \cdot \,) e_\mu),\qquad v \in V_\mu.$$
Then $T_\mu$ is an isometric $K$-intertwining operator between $\pi_\mu$ and $\ell$, where $\ell$ is the left regular
representation of $K$ on $L^2 (\cB; \cL_\chi)$. Let 
\[L_\mu^2 (\cB; \cL_\chi) := \im (T_\mu).\]
We see that $\ell$ is an irreducible $K$-representation on $L^2 (\cB; \cL_\chi)$ and it leads to the decomposition
$$L^2 (\cB; \cL_\chi) \cong_K \bigoplus_{\mu \in \Lambda_\chi^+} L_\mu^2 (\cB; \cL_\chi).$$
For $\mu \in \Lambda_\chi^+$ we write 
\[J_\mu (\lambda) := J (\lambda) \big|_{L_\mu^2 (\cB; \cL_\chi)}.\]

\begin{lemma}
There is a meromorphic function $\eta_\mu: \g{a}_{\C}^\ast \to \C$ such that
\begin{equation}
\label{eq8}
J_\mu (\lambda) = \eta_\mu (\lambda)\; \mathrm{id}_{L_\mu^2 (\cB; \cL_\chi)}.
\end{equation}
\end{lemma}

\begin{proof}
{}From Theorem \ref{thm1} we have seen that $J (\lambda)$ is a $G$-intertwining operator between $\pi_\lambda$ and
$\pi_{-\lambda}^\theta$. As $K$-representations, $\pi_\lambda$ and $\pi_{-\lambda}^\theta$ agree with $\ell$. 
So $J(\lambda)$ is a $K$-intertwining operator between $(\ell, L^2 (\cB; \cL_\chi))$ and itself. 
Since the multiplicity of $\pi_\mu$ in 
$L^2 (\cB; \cL_\chi)$ is one, and each $L_\mu^2 (\cB; \cL_\chi)$ is irreducible, by Schur's Lemma, 
$J_\mu (\lambda) = c \cdot \mathrm{id}$ where $c$ is a constant depending on $\mu$ and $\lambda$. That is, we can define this
constant as $\eta_\mu (\lambda)$ by (\ref{eq8}).
\end{proof}

It follows that
$$J (\lambda) = (J_\mu (\lambda))_{\mu \in \Lambda_\chi^+} = (\eta_\mu (\lambda)\; \text{id})_{\mu \in \Lambda_\chi^+},$$
or simply, $J = (\eta_\mu)$. The set of spectral functions $\eta_\mu (\lambda)$ with $\mu$ running over $\Lambda_\chi^+$ 
forms the $(K, \chi)$-spectrum (or simply $K$-spectrum) of $J (\lambda)$:
$$\{\eta_\mu (\lambda)\; \mid\; \mu \in \Lambda_\chi^+\}.$$
The numbers $\eta_\mu (\lambda)$ are the eigenvalues of $J (\lambda)$ on the $K$-types. 
Our foremost goal is to determine $\eta_\mu (\lambda)$ explicitly.
If $\varphi \in L_\mu^2 (\cB; \cL_\chi)$, then $J (\lambda) \varphi = \eta_\mu (\lambda) \varphi$. If $\varphi$ is in
addition a spherical function of $\chi$-type (so that $\varphi (e) = 1$), then
\begin{equation}
\label{eq7}
J (\lambda) \varphi (e) = \eta_\mu (\lambda) \varphi (e) = \eta_\mu (\lambda).
\end{equation}
The equation (\ref{eq7}) provides us a nice formula to compute the eigenvalues $\eta_\mu (\lambda)$ for $\mu \in \Lambda_\chi^+$.

Fix a $H_0 \in \g{a}$ such that $\ad (H_0)$ has eigenvalues $0$, $1$, and $-1$. 
We normalize the $G$-invariant $\R$-bilinear form $\ip{\cdot}{\cdot}$ on $\g{g}$ so that $\ip{H_0}{H_0} = 1$.
Denote by the same symbol the $\C$-bilinear extension to $\g{g}_\C$.   
The positive definite Laplace operator $\Omega$ on $\cB$ acts on smooth 
sections over $\cB$ the same as that of the Casimir element of $\mathfrak{k}$, up to a sign. Therefore,
$$\Omega \Big|_{L_\mu^2 (\cB; \cL_\chi)} = (\ip{\mu + 2 \rho_\mathfrak{k}}{\mu} + c)\; \mathrm{id}.$$
where $c$ is a constant not depending on $\mu$. Denote by 
\[\omega (\mu) = \ip{\mu + 2 \rho_\mathfrak{k}}{\mu}\, .\]
For $\mu \in \Lambda_\chi^+$, define $\Phi_\mu: L_\mu^2 (\cB; \cL_\chi) \otimes \g{s}_\C \to L^2 (\cB; \cL_\chi)$ by 
$$\Phi_\mu (\varphi \otimes X) = M (\ip{X}{\Ad ( \cdot ) H_0})\; \varphi$$
where $M ( \cdot )$ denotes the multiplication operator. Observe that $\Phi_\mu$ is $K$-intertwining and $\im \Phi$ is
$K$-invariant. So we can define a subset $S (\mu) \subset \Lambda_\chi^+$ by 
$$\mathrm{Im} \Phi_\mu \cong_K \bigoplus_{\nu \in S (\mu)} L_\nu^2 (\cB; \cL_\chi).$$
Let $\mathrm{pr}_\nu$ denote the orthogonal projection
$$\mathrm{pr}_\nu: L^2 (\cB; \cL_\chi) \longrightarrow L_\nu^2 (\cB; \cL_\chi).$$

\begin{lemma}
Let $\mu \in \Lambda_\chi^+$. Let $\nu \in S (\mu)$, $Y \in \g{s}_\C$, and $r \alpha \in \g{a}_{\C}^\ast$ 
where $\alpha \in \g{a}^\ast$ is so that $\alpha (H_0) =1$ and $r \in \C$. Let
\[\omega_{\nu\; \mu} (Y) := \mathrm{pr}_\nu \circ M (\ip{Y}{\Ad ( \cdot ) H_0})
\big|_{L_{\mu}^2 (\cB; \cL_\chi)} : L_{\mu}^2 (\cB; \cL_\chi) \longrightarrow L_{\nu}^2 (\cB; \cL_\chi).\]
We have
\begin{equation}
\label{eq9}
\mathrm{pr}_\nu \circ \pi_{r \alpha} (Y) \big|_{L_\mu^2 (\cB; \cL_\chi)} = \frac{1}{2} (\omega (\nu) - \omega (\mu)
+ 2 r) \omega_{\nu \mu} (Y).
\end{equation}
\end{lemma}

\begin{proof}
Note that 
\[\ip{\nu + 2 \rho_\mathfrak{k}}{\nu} + c - \ip{\mu + 2\rho_\mathfrak{k}}{\mu} - c = \omega (\nu) - \omega (\mu)\]
so that the constant $c$ is going to be canceled out. The rest is given by the proof of Corollary 2.6 in \cite{boo}.
\end{proof}

\begin{theorem}
Let $\mu \in \Lambda_\chi^+$, $\nu \in S (\mu)$ and $\lambda = r \alpha \in \g{a}_{\C}^\ast$. Then 
\begin{equation}
\label{eq10}
\frac{\eta_\nu (\lambda)}{\eta_\mu (\lambda)} = \frac{2 r - \omega (\nu) + \omega (\mu)}{2 r + \omega (\nu) -
\omega (\mu)}.
\end{equation}
The $K$-spectrum $\{\eta_\mu (\lambda)\}_{\mu \in \Lambda_\chi^+}$ and thus $J (\lambda)$ is uniquely determined by
(\ref{eq10}) and the normalization $\eta_{\mu^0} (\lambda)$, where $\mu^0$ is chosen as the smallest highest weight in 
$\Lambda_\chi^+$.
\end{theorem}

\begin{proof}
We apply $J (\lambda)$ to (\ref{eq9}) and get
\[J (\lambda) \left[\mathrm{pr}_\nu \circ \pi_{r \alpha} (Y) \big|_{L_\mu^2 (\cB; \cL_\chi)}\right] = 
\frac{1}{2} (\omega (\nu) - \omega (\mu) + 2 r) \eta_\nu (\lambda) \omega_{\nu \mu} (Y)\, .\]
On the other hand, using (\ref{eq6}) and the fact that $\theta (Y) = - Y$ for $Y \in \mathfrak{s}_\C$, 
$$J (\lambda) \circ \pi_\lambda (Y) = - \pi_{- \lambda} (Y) \circ J (\lambda)\, .$$
As $J (\lambda)$ commutes with $\mathrm{pr}_\nu$, 
$$J (\lambda) (\mathrm{pr}_\nu \circ \pi_{r \alpha} (Y)) = \mathrm{pr}_\nu \circ (- \pi_{- r \alpha} (Y) \circ J (\lambda)).$$
Therefore,
\[J (\lambda) \left[\mathrm{pr}_\nu \circ \pi_{r \alpha} (Y) \big|_{L_\mu^2 (\cB; \cL_\chi)}\right] = 
- \frac{1}{2} (\omega (\nu) - \omega (\mu) - 2 r) \omega_{\nu \mu} (Y) \eta_\mu (\lambda)\, .
\]
Since $\omega_{\nu \mu} (Y)$ is nonzero, we have
$$[\omega (\nu) - \omega (\mu) + 2 r] \eta_\nu (\lambda) = [- \omega (\nu) + \omega (\mu) + 2 r] \eta_\mu
(\lambda).$$
Thus, (\ref{eq10}) follows now. We can start to compute an initial eigenvalue $\eta_{\mu^0} (\lambda)$ and 
then reach the rest by the inductive formula (\ref{eq10}).
\end{proof}

\begin{remark}
We say an element $\mu \in \Lambda_\chi^+$ is smallest if the standard norm of $\mu$ is minimal.
\end{remark}

\section{The Grassmann Manifolds}
\label{grassmann}

\noindent
In Section \ref{sgo} we have seen that the spectral functions $\eta_\mu (\lambda)$ in (\ref{eq8}) are parametrized by the $\chi$-spherical 
representations of $K$. The main purpose of this section is to study 
the structure of line bundles on Grassmann manifolds $K/L$, to present the classification of the $\chi$-spherical representations of $K$, and to discuss 
the connection of $\chi$-spherical functions on $K$ to their corresponding hypergeometric functions in the Heckman-Opdam sense. 
These results will be used in Section \ref{spectrum}.

Let $\K$ denote one of the fields $\R$ or $\C$. Set $d = \dim_\R \K$, which is $1$ or $2$ for $\K = \R$ or $\C$, respectively. 
Let $G = \SL (n+1, \K)$.   
Denote by $\gp (\K)$ the Grassmann manifold of $p$-dimensional linear subspaces of $\K^{n+1}$. 
Since $\gp (\K) \cong \mathrm{G}_{n + 1 - p} (\K)$, we can assume $2 p \leq n + 1$. 
Let $q = n+1 - p$. Let $e_1 \dotsc, e_{n+1}$ be the standard basis of $\K^{n+1}$. Let 
$b_0 = \K e_1 \oplus \cdots \oplus \K e_p \in \gp (\K)$. 
Then $\gp (\K) = K \cdot b_0 \cong K / L = G / P = \cB$,
where $P$ is a maximal parabolic subgroup of $G$, and $K/L$ is a hermitian compact symmetric space (it is a symmetric $R$-space).

The classification of hermitian type Grassmann manifolds $K/L$ and their relevant information are given in the following table. 
Here, $t = \dim \g{b}$ and $s = \dim_\R K/L$. Refer to Section 2 in \cite{hol}
for more discussion.

{\small 
\begin{equation}
\label{t1} 
\centering
\begin{tabular}[c]{|c||c|c|c|c|c|c|}
\hline
\multicolumn{7}{|c|}{Classification of Grassmann manifolds of} \\
\multicolumn{7}{|c|}{Hermitian type: $G$ noncompact, $K$ compact}\\
\multicolumn{7}{|c|}{ } \\
\hline \hline
& $G$ & $K$ & $L$ & $t$ & $s$ & remark\\
\hline
\hline
$1$ & $\SL (n+1, \R)$ & $\mathrm{SO} (n+1)$ & $\mathrm{S} (\mathrm{O} (p) \times \mathrm{O} (q))$ &
$p$ & $2 q$ & $p=2, q \geq 3$\\
\hline
$2$ & $\SL (n+1, \C)$ & $\mathrm{SU} (n+1)$ & $\mathrm{S} (\mathrm{U} (p) \times \mathrm{U} (q))$
& $p$ & $2 p q$ & $q \geq p \geq 1$\\
\hline
\end{tabular}
\end{equation}}

In Table \ref{t1} the complex Grassmann manifold $\mathrm{SU} (n+1) / \mathrm{S} (\mathrm{U} (p) \times \mathrm{U} (q))$ 
is simply connected. The real Grassmann manifold $\mathrm{SO} (n+1) / \mathrm{S} (\mathrm{O} (p) \times \mathrm{O} (q))$
is not simply connected, but which has a simply connected covering 
\[\mathrm{SO} (n+1) / (\mathrm{SO} (p) \times \mathrm{SO} (q)).\]

{}From now on we specialize our work to cases in Table \ref{t1}. Define an invariant $\R$-invariant form on $\g{g}$ by
\begin{equation}
\label{eq:invform}
\langle X,\; Y\rangle := \frac{n+1}{p q}\; \re (\Tr (X Y)).
\end{equation}
The form (\ref{eq:invform}) agrees with the invariant form in Section \ref{sgo}. 
For each case in Table \ref{t1}, we can write down explicit expressions of $\theta$,
$H_0$, $\g{s}$, $\g{m}$, $\g{a}$, $\g{n}$, $\b{\g{n}}$, $M A$, $N$, $P$, $\tau$, $\g{q}$, $\g{b}$, and so on.
They were given in Section 4 and Section 5 of \cite{op}. 
In particular, we mention here what we need to use:
\begin{equation}
\label{eq28}
M A = \left\{ 
\begin{pmatrix}
a & 0\\
0 & b 
\end{pmatrix}\; \Big|\; a \in \GL (p, \K),\, b \in \GL (q, \K),\, \det a  \det b = 1 \right\}\, ,
\end{equation}
and the elements of $\exp (\g{b})$,
\begin{equation}
\label{eq:B}
\exp Y (\mathbf{t}) = 
\begin{pmatrix}
\mathrm{diag} (\cos t_1, \dotsc, \cos t_p) & 0 & - \mathrm{diag} (\sin t_1, \dotsc, \sin t_p)\\
0 & I_{q - p} & 0\\
\mathrm{diag} (\sin t_1, \dotsc, \sin t_p) & 0 & \mathrm{diag} (\cos t_1, \dotsc, \cos t_p)
\end{pmatrix},
\end{equation}
where $Y (\mathbf{t}) \in \g{b}$ and $\mathbf{t} = (t_1, \dotsc, t_p)^T \in \R^p$ (cf. \cite[p.286]{op}). 
Let $\{\eps_1, \dotsc, \eps_p\}$ be an orthogonal basis of $i \g{b}^\ast$ with $\eps_j (Y (\mathbf{t})) = i t_j$.
We identify the element 
\[\sum_{j=1}^p \lambda_j \eps_j\; \stackrel{\cong}{\longleftrightarrow}\; (\lambda_1, \dotsc, \lambda_p).\]

Next, we classify nontrivial characters of $L$ which 
determine nontrivial line bundles over $K/L$, and the extensions of these characters to $M$.

\begin{proposition}
In case $\K = \R$, the only nontrivial character of $L$ is given by 
\begin{equation}
\label{eq:chi-nonc}
\chi\, 
\begin{pmatrix}
A & 0\\
0 & B
\end{pmatrix} \longmapsto \det A,\quad A \in \mathrm{O} (p),\, B \in \mathrm{O} (q).
\end{equation}
In case $\K = \C$, the group of nontrivial characters of $L$ is parametrized by $\Z$. Precisely, fix $Z \in \g{z} (\g{l})$ as in 
\cite[p.283, (3.1)]{sch}. Let $l \in \Z$ and define $\chi_l: L \to \mathbb{T}$ by 
\begin{equation}
\label{eq:chi-c}
\chi_l (\exp (t Z) k) = e^{i l t},\quad t \in \R\; \mathrm{and}\; k \in [L, L].
\end{equation}
Then $\chi_l$ is a well defined character on $L$. If $\chi$ is a character on $L$, then there is a unique $l \in \Z$ such that 
$\chi = \chi_l$. 
\end{proposition}

\begin{proof}
In case $\K = \R$, $L = \mathrm{S} (\mathrm{O} (p) \times \mathrm{O} (q))$ with $p=2$.  
So a character on $L$ can be determined by either a character on $\mathrm{O} (2)$ or a character on $\mathrm{O} (q)$. 
We know that $\mathrm{O} (n)$ is not abelian for $n > 1$. A simple observation shows that the only nontrivial character 
of $\mathrm{O} (2)$ is given by $x \mapsto \det x$, $x \in \mathrm{O} (2)$. Therefore, the only nontrivial character 
of $L$ is given by (\ref{eq:chi-nonc}). In case $\K = \C$, the formula (\ref{eq:chi-c}) follows from 
Proposition 3.4 in \cite{sch}.
\end{proof}

Recall that a character on $L$ can be uniquely extended to a character on $M$ by (\ref{eq:chi-M}), and 
conversely a character on $M$ restricted to $L$ is exactly a character on $L$. 
For an element $g \in \overline{N} M A N$, 
\begin{equation}
\label{eq26}
g = 
\begin{pmatrix}
X & Y\\
V & W
\end{pmatrix} \in G,\qquad X \in \mathrm{GL} (p, \K)\, .
\end{equation}

\begin{proposition}
If $\chi$ is a (nontrivial) character on $M$, then 
\begin{equation}
\label{eq38}
\chi (m (g)) = \left(\frac{\det X}{|\det X|}\right)^j
\end{equation}
for some $j \in \mathbb{Z}$. The integer $j$ is determined as: $j = 1$ if $\mathbb{K} = \mathbb{R}$; and $j = l$ if 
$\mathbb{K} = \mathbb{C}$ with $l$ being the integers parametrizing the characters of the connected component $L_0$ of $L$
(the same $l$ as in (\ref{eq:chi-c})).
\end{proposition}

\begin{proof}
Let $g \in G$ be as in (\ref{eq26}). We have 
\[
\begin{pmatrix}
X & 0\\
0 & W - V X^{-1} Y
\end{pmatrix} \in M A\]
satisfying the condition in (\ref{eq28}). Let $\chi$ be a character of $M$. Then (\ref{eq38}) follows with $j$ to be  
determined. 
(1) In case $\mathbb{K} = \mathbb{R}$: if we apply $\chi$ to elements in $L$ then (\ref{eq38}) should coincide with 
(\ref{eq:chi-nonc}). It gives $j = 1$.
(2) In case $\mathbb{K} = \mathbb{C}$ we write 
\[Z = 
\begin{pmatrix}
Z_1 & \\
& Z_2
\end{pmatrix},\quad \exp (t Z) = 
\begin{pmatrix}
e^{t Z_1} & \\
& e^{t Z_2}
\end{pmatrix} \in \mathrm{S} (\mathrm{U} (p) \times \mathrm{U} (q))\]
where $Z_1$ is a $p \times p$ block and $Z_2$ is a $q \times q$ block on the diagonal with $|\det e^{t Z_1}| = 1$
(recall \cite[p.283, (3.1)]{sch} for the construction of $Z$ in details). A direct calculation shows that 
$\mathrm{Tr} (Z_1) = i$ (a pure imaginary number). So 
\[\chi (\exp (t Z)) = \left(\frac{\det (\exp (t Z_1))}{|\det (\exp (t Z_1))|}\right)^j = 
\det (\exp (t Z_1))^j = e^{j t \mathrm{Tr} (Z_1)}\]
for some $j \in \mathbb{Z}$. The restriction of $\chi$ to $L$ is a (nontrivial) character on $L$, which should 
match (\ref{eq:chi-c}), that is, for some $l \in \Z$,
\[\chi (\exp (t Z)) = \chi_l (\exp (t Z)) = e^{i l t},\quad \forall t \in \mathbb{R}.\]
Therefore,
\[i j t = j t \mathrm{Tr} (Z_1) = i l t, \quad \forall t \in \mathbb{R}.\]
This leads to $j = l$.
\end{proof}

The next two lemmas give the description of root system $\Sigma$ and the classification of $\chi$-spherical 
representations of $K$.

\begin{lemma}
\label{lem:Sigma}
The set of roots of $\g{b}_\C$ in $\g{k}_\C$ is given by 
$$\Sigma = \{\pm \eps_i (1 \leq i \leq p),\quad \pm \eps_i \pm \eps_j (1 \leq i \ne j \leq p),\quad \pm 2 \eps_i (1 \leq i
\leq p)\}$$
with multiplicities $d (q - p)$, $d$, and $d - 1$, respectively, where $\pm$ signs are independent.
\end{lemma}

\begin{proof}
This is a well-known result. Refer to Moore \cite[Theorem 5.2]{mo} and \cite[p.528, 532]{h1}.
\end{proof}

{}From now on we fix a positive root system 
\[\Sigma^+ = \{\eps_i (1 \leq i \leq p),\quad \eps_i \pm \eps_j (1 \leq i < j \leq p),\quad 2 \eps_i (1 \leq i \leq p)\}\, .\]

\begin{theorem}
\label{thm5}
In case $\K = \R$, the sublattice $\Lambda_\chi^+$ of the set of highest weights of irreducible $\chi$-spherical 
representations (so $\Lambda_\chi^+ \cong \wh{K}_\chi$) is given by
\begin{equation}
\label{LamPlus1}
\{\mu \in i \mathfrak{b}^\ast\; \Big|\; \mu = \sum_{i = 1}^p \mu_j \eps_j,\;
\mu_1 \geq \dotsb \geq \mu_p \geq 0\; \text{and}\; \mu_j\; \text{odd for all}\; j\}.
\end{equation}
In case $\K = \C$, the set $\Lambda_\chi^+ = \Lambda_{\chi_l}^+ =: \Lambda_l^+$ (with $l \in \Z$)
of the highest weights of irreducible $\chi_l$-spherical representations is given by
\begin{equation}
\label{LamPlus}
\{\mu \in i \mathfrak{b}^\ast\; \Big|\; \mu = \sum_{i = 1}^p \mu_j \eps_j,\; \mu_i - \mu_j \in 2 \mathbb{Z}^+,
\mu_p \in |l| + 2 \mathbb{Z}^+,  1 \leq i < j \leq p\}.
\end{equation}
\end{theorem}

\begin{proof}
(1) In case $\mathbb{K} = \mathbb{R}$, let $\chi$ be given by (\ref{eq:chi-nonc})
and $\mu \in \Lambda_\chi^+$. We have 
\[\mu_j = \frac{\ip{\mu}{\eps_j}}{\ip{\eps_j}{\eps_j}} \geq 0,\; \forall j\]
and also
\[\mu_i - \mu_j = \frac{\ip{\mu}{\eps_i - \eps_j}}{\|\eps_j\|^2} \geq 0,\quad i < j,\]
whence $\mu_i \geq \mu_j$. Consider an element in $L$ (but not in $L_0$), e.g., let $x = \exp (Y (\mathbf{t}))$
given by (\ref{eq:B}) with $\mathbf{t} = (t_1, t_2)$, $t_1 = \pi$ and $t_2 = 0$. We have 
$\chi (x) = -1$. Note that 
\[\pi_\mu (x) = \exp (\mu (Y (\mathbf{t})) = \exp (\sum_{j=1}^p \mu_j \eps_j (Y (\mathbf{t})) 
= \exp (\sum_{j=1}^p \mu_j \cdot i \cdot t_j) = e^{i \pi \mu_1}.\]
Let $u \in V_\mu$ be a highest weight vector of the weight $\mu$ such that $\ip{u}{e_\mu} \ne 0$. Then 
\[-1 \cdot \ip{u}{e_\mu} = \ip{ u}{ \chi (x)^{-1} e_\mu} = \ip{ u}{ \pi_\mu (x^{-1}) e_\mu } = \ip{ \pi_\mu (x) u}{ e_\mu} 
= e^{i \pi \mu_1} \ip{ u}{ e_\mu }\]
which implies that $e^{i \pi \mu_1} = -1$, i.e., $\mu_1$ is an odd integer. In a similar fashion, we can show all $\mu_j$ 
($1 \leq j \leq p$) are odd positive integers. Hence, (\ref{LamPlus1}) follows.

(2) In case $\K = \C$, the classification (\ref{LamPlus}) follows from Proposition 7.1 and Theorem 7.2 in \cite{sch}. Also 
see Theorem 3.1 in \cite{hol}. 
\end{proof}

Let $\mathcal{O}_s^+ = \{\eps_j\}$ and 
\[\rho_s = \frac{1}{2} \sum_{\alpha \in \mathcal{O}_s^+} \eps_j = \frac{1}{2} (\eps_1 + \cdots + \eps_p).\] 
It follows from Proposition 3.3 in \cite{hol} that $\Lambda_l^+ = \Lambda_0^+ + 2 |l| \rho_s$ where
\begin{equation}
\label{eq31}
\Lambda_0^+ = \{\mu \in i \g{b}^\ast\; \Big|\; \frac{\langle \mu, \alpha\rangle}{\langle \alpha, \alpha \rangle} \in
\Z^+,\; \forall \alpha \in \Sigma^+\}.
\end{equation} 
Let $k = 1$ for the case $\K = \R$, and $k = |l|$ (with $l \in \Z$) for the case $\K = \C$. 
Set 
\begin{equation}
\label{eq32}
\mu^0 = 2k \rho_s = k (\eps_1 + \cdots + \eps_p).
\end{equation}
It is easy to see from (\ref{LamPlus1}), (\ref{LamPlus}), and (\ref{eq31}) that $\mu^0 \in \Lambda_\chi^+$ is 
the smallest\footnote{For an element $\mu \in \g{b}_{\C}^\ast$ we say that $\mu$ is smallest if its standard norm $\|\mu\|$ is minimal.}
element in $\Lambda_\chi^+$ for the cases $\K=\R$ and $\K=\C$, respectively.

To close this section, we review some useful properties about $\chi$-spherical functions and their connections
to the Heckman-Opdam hypergeometric functions (cf. Part I, Chapters 4, 5 in \cite{hs}), and make an observation on 
these functions, particularly evaluated at the smallest element in $\Lambda_\chi^+$.

Let $K^d$ denote the noncompact dual of $K$. Then $(K^d, L_0)$ is a symmetric pair of the noncompact type. 
Let $\mathbf{X}=K^d/L_0$.
Let $\mu \in \Lambda_\chi^+$ and $\pi_\mu$ the $\chi$-spherical unitary representation of $K$ with the weight $\mu$.

(1) Consider the case $\K = \R$ with the character $\chi$ on $L$ given by (\ref{eq:chi-nonc}). Since $\chi |_{L_0} = 1$,  
$e_\mu$ is a unit $L_0$-fixed vector and thus $\pi_\mu$ is a $L_0$-spherical representation.
Hence the $\chi$-spherical functions $\psi_\mu$ on $K$ is also a $L_0$-spherical function. The core property is that $\psi_\mu$
are connected to the $\chi$-spherical functions $\varphi_{\lambda_{\g{k}}}$ 
on $K^d$ by holomorphic continuation where $\lambda_{\g{k}} \in \g{b}_{\C}^\ast$ (cf. Section 4 in \cite{hol}), precisely, 
\[\psi_\mu = \varphi_{\mu + \rho_{\g{k}}}\]
where $\rho_{\g{k}}$ is given by (\ref{eq:rhok}). We are interested in $\varphi_{\mu + \rho_{\g{k}}}$ at $\mu = \mu^0$. 
Note that $\mu^0 + \rho_{\g{k}} = 2 \rho_s + \rho_\mathfrak{k}$.

\begin{lemma}
In case $\K = \R$ let $b \in \exp (\mathfrak{b})$ (so $b$ has the form (\ref{eq:B})). We have
\begin{equation}
\label{eq23}
\varphi_{2 \rho_s + \rho_\mathfrak{k}} (b) = \prod_{j=1}^p \cos t_j.
\end{equation}
\end{lemma}

\begin{proof}
Let $L (m)$ be the radial part of the Laplace-Beltrami operator on $\mathbf{X}$ associated with the root 
system $\Sigma$ and the multiplicity $m$:
\[L (m) = \sum_{j=1}^p \frac{\partial^2}{\partial t_j^2} + 
\sum_{\alpha \in \Sigma^+} m_\alpha \frac{1+e^{-2 \alpha}}{1-e^{-2 \alpha}} \partial_\alpha.\]
Recall Lemma \ref{lem:Sigma} for $\Sigma^+$ and $m_\alpha$. In case $\K = \R$, $d = 1$ and $p = 2$. 
We apply $L (m)$ to the function $\prod_{j=1}^2 \cos t_j$ and get 
\[L (m) (\cos t_1 \cos t_2) = (- 2 q) (\cos t_1 \cos t_2).\]
On the other hand, we apply $L (m)$ to the function $\varphi_{2 \rho_s + \rho_\mathfrak{k}}$:
\begin{eqnarray*}
L (m) \varphi_{2 \rho_s + \rho_\mathfrak{k}} (b) & = & -[\langle 2 \rho_s + \rho_\mathfrak{k},\; 2 \rho_s + \rho_\mathfrak{k}\rangle - 
\langle \rho_\mathfrak{k},\; \rho_\mathfrak{k}\rangle]\; \varphi_{2 \rho_s + \rho_\mathfrak{k}} (b)\\
& = & - \langle 2 \rho_s + 2 \rho_\mathfrak{k},\; 2 \rho_s \rangle \varphi_{2 \rho_s + \rho_\mathfrak{k}} (b).
\end{eqnarray*}
Note that $2 \rho_s \cong (1, 1)$ and $2 \rho_\mathfrak{k} \cong (q, q-2)$. Thus, 
\[- \langle 2 \rho_s + 2 \rho_\mathfrak{k},\; 2 \rho_s \rangle = - \langle (q+1,\; q-1),\; (1, 1) \rangle = -2 q.\]
Hence, 
\[L (m) \varphi_{2 \rho_s + \rho_\mathfrak{k}} = (-2 q) \varphi_{2 \rho_s + \rho_\mathfrak{k}}.\]
Since both $\varphi_{2 \rho_s + \rho_\mathfrak{k}} (b)$ and $\cos (t_1) \cos (t_2)$ share the same eigenvalue, and since 
the spherical functions are the unique normalized solutions of the system of differential equations, originated from the algebra of $K^d$-invariant 
differential operators on $\mathbf{X}$ (cf. \cite[Chapter IV, \S 2]{h2}), we see that 
\[\varphi_{2 \rho_s + \rho_\mathfrak{k}} (b) = \cos (t_1) \cos (t_2).\]
\end{proof}

(2) For the case $\K = \C$ the characters $\chi = \chi_l$ on $L$ are given by (\ref{eq:chi-c}) with $l \in \Z$. 
The $\chi_l$-spherical functions $\psi_{\mu, l}$
on $K$ are connected to the $\chi_l$-spherical functions $\varphi_{\lambda_{\g{k}}, l}$ on $K^d$ by 
holomorphic continuation (also cf. \cite{hol}), 
\[\psi_{\mu, l} = \varphi_{\mu + \rho_{\g{k}}, l}.\]
Let $m = (m_s, m_m, m_l)$ where $m_s, m_m, m_l$ denote the multiplicities of short, medium, and long roots in 
$\Sigma$, respectively. The multiplicity parameters 
\begin{equation}
\label{mpn}
m_\pm (l) = (m_s \mp 2 |l|, m_m, m_l \pm 2 |l|)
\end{equation}
(the $\pm$ signs are both valid) are not necessarily positive. By applying the multiplicities (\ref{mpn}) to (\ref{eq:rhok}) 
and regrouping terms (cf. Proposition 2.10 in \cite{ho}), we get 
\begin{equation}
\label{eq12}
\rho_\mathfrak{k} (m_\pm (l)) = \rho_\mathfrak{k} \pm 2 |l| \rho_s,\qquad \rho_\mathfrak{k} = \rho_\mathfrak{k} (m).
\end{equation}

The $\chi_l$-spherical functions $\varphi_{\lambda_{\g{k}}, l}$ are
linked to the Heckman-Opdam hypergeometric functions $F$ associated to $\Sigma$ and the multiplicity $m_\pm (l)$ 
(cf. Theorem 5.2.2, part I, \cite{hs}), that is, 
\begin{equation}
\label{eq11}
\varphi_{\lambda_\mathfrak{k}, l} |_{\exp (i \g{b})} = \eta_l^\pm\; F (\lambda_\mathfrak{k}, m_\pm (l);\;
\cdot\; ),\;\quad \lambda_\g{k} \in \g{b}_\C^\ast
\end{equation}
where
\[\eta_l^\pm = \prod_{\alpha \in \mathcal{O}_s^+} \left(\frac{e^\alpha + e^{- \alpha}}{2}\right)^{\pm 2 |l|}.\]
Applying $L (m)$ to the hypergeometric function (associated to a multiplicity parameter $m$), we get
$$L (m) F (\lambda_\mathfrak{k}, m;\; \cdot\; ) = (\langle \lambda_\mathfrak{k}, \lambda_\mathfrak{k}\rangle - \langle
\rho_\mathfrak{k} (m), \rho_\mathfrak{k} (m)\rangle) F (\lambda_\mathfrak{k}, m;\; \cdot\; ),\quad
\lambda_\mathfrak{k} \in \mathfrak{b}_{\C}^\ast.$$
Write $m (l) = m_+ (l)$ and $\rho_{\g{k}, l} = \rho_\mathfrak{k} (m (l)) \in \g{b}_\C^\ast$. 
It follows from (\ref{eq32}) and (\ref{eq12}) that 
\[\mu^0 + \rho_{\g{k}} = 2 |l| \rho_s + \rho_{\g{k}} = \rho_{\g{k}, l}.\]

\begin{lemma}
We have 
\begin{equation}
\label{eq13}
\varphi_{\rho_{\g{k}, l},\, l} (x) = \eta_l^+ (x), \quad x \in \exp \g{b}_\C.
\end{equation}
\end{lemma}

\begin{proof}
Note that 
\[L (m (l))\; F (\rho_{\g{k}, l}, m (l);\; \cdot\; ) = (\ip{\rho_{\g{k}, l}}{\rho_{\g{k}, l}} - \ip{\rho_{\g{k}, l}}
{\rho_{\g{k}, l}})\; F (\rho_{\g{k}, l}, m (l);\; \cdot\; ) = 0\]
which implies that $F (\rho_{\g{k}, l}, m (l);\; \cdot\; )$ equals some constant function. But since we normalize
$\varphi_{\lambda_\mathfrak{k}, l} (e) = 1$, it follows from (\ref{eq11}) that $F (\rho_{\g{k}, l}, m (l);\; \cdot\; ) = 1$. 
Hence (\ref{eq13}) follows for $x \in \exp i \g{b}$. Because 
$\exp i \g{b}$ is a totally real form of $\exp \g{b}_\C$, (\ref{eq13}) holds for $x \in \g{b}_\C$.
\end{proof}

The formulas (\ref{eq23}) and (\ref{eq13}) are of crucial importance 
when we compute an initial (in particular, the smallest) term in the spectrum of $J (\lambda)$ in Theorem \ref{thm3}.

\section{The $K$-spectrum of the $\ct$ transform}
\label{spectrum}

\noindent
In this section we will introduce the $\ct$ transform on smooth sections of nontrivial line
bundles over $\gp (\K) = K/L$ (for $\K = \R$ or $\C$ as we discussed in Section \ref{grassmann})  
and prove that it coincides with the intertwining operator $J (\lambda)$. 
Then we apply the spectrum generating method outlined in Section \ref{sgo} to determine the $K$-spectrum of the $\ct$ transform.

Let $x, y \in \gp (\K)$. We can view $x$ as a $d p$-dimensional real vector space. Let $E \subset x$ be a
convex subset containing the origin and $\mathrm{Vol}_\R (E) = 1$ (the volume). 
Let $\mathrm{pr}_y: \K^n \to y$ be the orthogonal projection onto $y$. Define
$$|\mathrm{Cos} (x, y)| := \mathrm{Vol}_\R (\mathrm{pr}_y (E))^{1/d}.$$
The definition of $\mathrm{Cos}$ is independent of the choice of $E$. Moreover,
\[|\mathrm{Cos} (x, y)| = |\mathrm{Cos} (y, x)| \geq 0.\]
Refer to Section 3.2 in \cite{opr} for more information about this $\mathrm{Cos}$-function. We identify $\g{a}_{\C}^\ast$ with $\C$ by 
$$\lambda \longmapsto \frac{n+1}{p q} \lambda (H_0)\qquad \text{with the inverse}\qquad z \longmapsto z \frac{p q}{n+1} \alpha.$$
Therefore, 
\[\rho \cong \frac{d (n+1)}{2} \in \R.\]
Theorem 4.2 and Corollary 4.4 in \cite{op} give the following important result:

\begin{theorem}
\label{thm2}
For $\lambda \in \C$, $x = k \cdot b_0 \in \cB$, and $y = h \cdot b_0 \in \cB$, we have
$$\alpha (k^{-1} h)^\lambda = |\mathrm{Cos} (x, y)|^\lambda,\quad k, h \in K.$$
\end{theorem}

We use the convention that 
\[\mathrm{Cos} (k, h) := \mathrm{Cos} (k \cdot b_0,\; h \cdot b_0).\] 
For $\re \lambda > \rho$, we define the
$\ct$ transform $\cl: C^\infty (\cB; \cL_\chi) \to C^\infty (\cB; \cL_\chi)$ by 
\begin{equation}
\label{eq:ct}
\cl f (k) = \int_K |\mathrm{Cos} (k, h)|^{\lambda - \rho} \chi (m (k^{-1} h)) f (h)\, d h.
\end{equation}
This is a higher-rank analogue of (\ref{eq:cos2}). If $p=1$ and $\chi$ is a trivial character, then (\ref{eq:ct}) 
matches (\ref{eq:cos2}) with a $\rho$-shift in the power $\lambda$.

Theorem \ref{thm2} implies that $J (\lambda)$ coincides with $\mathcal{C}^\lambda$ immediately. Also, the proof of Theorem
\ref{thm1} shows that
$\cl$ is an intertwining operator. Furthermore, the $\ct$ transform (\ref{eq:ct}) extends to a meromorphic family of intertwining 
operators $\cl$ for $\lambda \in \g{a}_\C^\ast$.

Next we will determine the $K$-spectrum of $J (\lambda)$ for $\cB = \gp (\K)$, which determines the
$K$-spectrum of the $\ct$ transform (\ref{eq:ct}).

Recall (\ref{LamPlus1}) and (\ref{LamPlus}) for the description of $\Lambda_\chi^+$. Then accordingly 
there is a realization of the set $S (\mu)$ (defined in Section \ref{sgo}) for $\mu \in \Lambda_\chi^+$, and 
hence (\ref{eq10}) yields an explicit form (\ref{recur}). 
The proofs are similar to those of Lemma 5.4 and Lemma 5.5 in \cite{op} with 
$\Lambda^+ (\cB)$ replaced by $\Lambda_\chi^+$.

\begin{lemma}
Let $\mu = (\mu_1, \dotsc, \mu_p) \in \Lambda_\chi^+$. Then 
$$S (\mu) = (\{\mu \pm 2 \eps_j\; \mid\; j = 1, \dotsc, p\} \cup \{\mu\}) \cap \Lambda_\chi^+.$$
These representations occur with multiplicity one.
\end{lemma}

\begin{lemma}
Let $\mu = (\mu_1, \dotsc, \mu_p) \in \Lambda_\chi^+$ and $\lambda \in \C \cong \g{a}_{\C}^\ast$. Then
\begin{equation}
\label{recur}
\frac{\eta_{\mu + 2 \eps_j} (\lambda)}{\eta_\mu (\lambda)} = \frac{\lambda - \mu_j - \rho + d (j - 1)}{\lambda +
\mu_j + \rho - d (j - 1)}.
\end{equation}
\end{lemma}

The formula (\ref{recur}) is the reason for the term spectrum generating, and ensues a recursive calculation of
the spectral functions $\eta_\mu$.
We use (\ref{recur}) as an induction to determine all $\eta_\mu (\lambda)$ starting from the smallest one. 
We will see that $\eta_\mu (\lambda)$ are given by rational functions of Siegel $\Gamma$-functions.

Let us introduce the Siegel (or Gindikin) $\Gamma$-function $\Gamma_{p, d}: \C^p \to \C$ by 
$$\Gamma_{p, d} (\lambda) = \prod_{j = 1}^p \Gamma \left(\lambda_j - \frac{d}{2} (j - 1)\right).$$
Then $\Gamma_{p, d} (\lambda)$ is meromorphic with singularities on the union of the hyperplanes 
\[\{\lambda \in \C^p \mid \lambda_j = - n + \frac{d}{2} (j - 1)\},\]
where $j = 1, \dotsc, p$ and $n \in \Z^+$. If $\lambda \in \C$ we identify 
$\lambda \mapsto (\lambda, \dotsc, \lambda) \in \C^p$.

\begin{theorem}
Let $\lambda \in \C$. Then there is a meromorphic function $G (\lambda)$ such that for all $\mu \in \Lambda_\chi^+$ we
have 
\begin{equation}
\label{eq21}
\eta_\mu (\lambda) = (-1)^{|\mu|/2} G (\lambda) \frac{\Gamma_{p, d} (\frac{1}{2} (- \lambda + \rho
+ \mu))}{\Gamma_{p, d} (\frac{1}{2} (\lambda + \rho + \mu))}
\end{equation}
where $|\mu| = \sum_{j=1}^p \mu_j$.
\end{theorem}

\begin{proof}
The proof uses the recursive relation (\ref{recur}) and the property of Gamma function: $\Gamma (z+1) = z \Gamma (z)$.
\end{proof}

The function $G (\lambda)$ will be determined in (\ref{eq29}). For that, let us first compute an initial term in the spectrum 
of $J (\lambda)$.

\begin{theorem}
\label{thm3}
Let $\mu^0 \in \Lambda_\chi^+$ be as in (\ref{eq32}). We have
\begin{equation}
\label{eq19}
\eta_{\mu^0} (\lambda) = \frac{\Gamma_{p, d} (d (n+1) / 2)}{\Gamma_{p, d} (d p / 2)} \frac{\Gamma_{p, d}
(\frac{1}{2} (\lambda - \rho + k + d p))}{\Gamma_{p, d} (\frac{1}{2} (\lambda + \rho  + k))}.
\end{equation}
\end{theorem}

\begin{proof}
(1) In case $\K = \C$, $\chi = \chi_l$ for some $l \in \Z$ (cf. (\ref{eq:chi-c})) and $\Lambda_{\chi}^+ = \Lambda_l^+$
is given by (\ref{LamPlus}). From (\ref{eq7}) we have for $\mu \in \Lambda_l^+$,
\begin{equation}
\label{eq30}
\eta_\mu (\lambda) = J (\lambda) \psi_{\mu, l} (e) = \int_K \alpha (h)^{\lambda - \rho} \chi_l (m (h))
\psi_{\mu, l} (h)\, d h\, .
\end{equation}
We do not expect a direct evaluation of the integral when plugging $\mu^0$ into (\ref{eq30}) to compute $\eta_{\mu^0} (\lambda)$. 
Instead we reduce it to the case \cite{op}. 
In view of (\ref{eq22}), (\ref{eq18}), (\ref{eq12}), and (\ref{eq13}), there is a constant $C$ such that 
\begin{eqnarray*}
\eta_{\mu^0} (\lambda) & = & C \int_{\exp \g{b}} \alpha (b)^{\lambda -\rho} \chi_l (m (b)) 
\varphi_{\mu^0 + \rho_\mathfrak{k},\; l} (b) \delta (b)\, d b\\
& = & C \int_{\exp \g{b}} \alpha (b)^{\lambda -\rho} \chi_l (m (b)) \eta_l^+ (b) \delta (b)\, d b.
\end{eqnarray*}
{}From Lemma 5.8 in \cite{op} we have for $\lambda \in \C$,
\begin{equation}
\label{eq15}
\alpha (b)^\lambda = \alpha (\exp (Y (\mathbf{t})))^\lambda = \prod_{j = 1}^p |\cos (t_j)|^\lambda.
\end{equation}
Moreover, by (\ref{eq:B}) and (\ref{eq38}), 
\begin{equation}
\label{eq16}
\chi_l (m (b)) = \frac{\prod_{j=1}^p (\cos t_j)^l}{\prod_{j=1}^p |\cos t_j|^l}\, .
\end{equation}
Since $b^{\eps_j} = e^{\eps_j (Y (\mathbf{t}))} = e^{i t_j}$, we have 
\begin{equation}
\label{eq17}
\eta_l^+ (b) = \prod_{\alpha \in \mathcal{O}_s^+} \left(\frac{b^\alpha + b^{- \alpha}}{2}\right)^{|l|} = \prod_{j = 1}^p
(\cos t_j)^{|l|}.
\end{equation}
It follows from (\ref{eq15}), (\ref{eq16}), and (\ref{eq17}) that 
$$\alpha (b)^{\lambda -\rho} \chi_l (m (b)) \eta_l^+ (b) = \prod_{j = 1}^p |\cos t_j|^{\lambda - \rho + |l|}.$$
This gives 
\begin{equation}
\label{eq:eta0}
\eta_{\mu^0} (\lambda) = C \int_{\exp \g{b}} \alpha (b)^{\lambda -\rho + |l|} \delta (b) d b = \int_K
\alpha (h)^{\lambda - \rho + |l|}\, d h.
\end{equation}
(2) In case $\mathbb{K} = \mathbb{R}$, $\chi$ is given by (\ref{eq:chi-nonc}), and the sublattice $\Lambda_\chi^+$ is 
given by (\ref{LamPlus1}). Then
\[\eta_{\mu^0} (\lambda) = C \int_{\exp \mathfrak{b}} \alpha (b)^{\lambda - \rho} \chi (m (b)) \varphi_{2 \rho_s + \rho_\mathfrak{k}} (b)
 \delta (b)\, d b.\]
It follows from (\ref{eq:B}), (\ref{eq38}), (\ref{eq15}), and (\ref{eq23}) that
\begin{equation}
\label{eq:eta1}
\eta_{\mu^0} (\lambda) = C \int_{\exp \g{b}} \alpha (b)^{\lambda -\rho + 1} \delta (b)\, d b = \int_K
\alpha (h)^{\lambda - \rho + 1}\, d h.
\end{equation}
The formulas (\ref{eq:eta0}) and (\ref{eq:eta1}) can be written as 
\begin{equation}
\label{eq:eta}
\eta_{\mu^0} (\lambda) = \int_K \alpha (h)^{\lambda - \rho + k}\, d h
\end{equation}
where $k = 1$ for $\mathbb{K} = \mathbb{R}$, and $k = |l|$ for $\mathbb{K} = \mathbb{C}$.
Note that (\ref{eq:eta}) is simply the term \cite[(5.5)]{op} with a shifted factor. So 
$$\eta_{\mu^0} (\lambda) = C \prod_{j = 1}^p \frac{\Gamma (\frac{1}{2} (\lambda - \rho + k + d p - d (j -
1)))}{\Gamma (\frac{1}{2} (\lambda + \rho + k - d (j - 1)))}$$
with $C$ to be determined. 
We may compute $C$ by evaluating $\eta_{\mu^0} (\lambda)$ at $\lambda = \rho - k$ since 
$\eta_{\mu^0} (\rho - k) = 1$. Precisely,
\begin{eqnarray*}
1 & = & C \prod_{j = 1}^p \frac{\Gamma (\frac{1}{2} (\rho - k - \rho + k + d p - d (j - 1)))}{\Gamma
(\frac{1}{2} (\rho - k + \rho + k - d (j - 1)))}\\
& = & C \prod_{j=1}^p \frac{\Gamma (\frac{1}{2} (d p - d (j - 1)))}{\Gamma (\frac{1}{2} (2 \rho - d (j - 1)))}\\
& = & C \frac{\Gamma_{p, d} (d p / 2)}{\Gamma_{p, d} (\rho)}
\end{eqnarray*}
Using the fact that $\rho = d (n+1) / 2$, we have 
$$C = \frac{\Gamma_{p, d} (d (n+1) / 2)}{\Gamma_{p, d} (d p / 2)}.$$
The desired formula (\ref{eq19}) follows immediately.
\end{proof}

\begin{remark}
In the integral (\ref{eq:eta}), the factor $\alpha (b)^{\lambda - \rho + k}$ might produce singularities. 
So we need the condition $\re (\lambda - \rho + k) > - 1$ for (\ref{eq:eta}) converges. That is to say,
(\ref{eq19}) holds for $\lambda \in \g{a}_\C^\ast (c_P)$. However it has a meromorphic continuation to $\g{a}_\C^\ast$,
as rational functions of $\Gamma$-factors. 
\end{remark}

\begin{theorem}
Let $\Lambda_\chi^+$ be the sublattice in Theorem \ref{thm5} parametrizing the $\chi$-spherical representations of $K$.
Let $\mu = (\mu_1, \dotsc, \mu_p) \in \Lambda_\chi^+$ and $\lambda \in \C$. The $K$-spectrum of the $\ct$ transform
$\cl$ (\ref{eq:ct}) is given by 
\begin{equation}
\label{eq20}
\eta_\mu (\lambda) = c (\mu) \frac{\Gamma_{p, d} (\frac{d}{2} (n+1))}{\Gamma_{p, d} (\frac{1}{2} d p)}
\frac{\Gamma_{p, d} (\frac{1}{2} (\lambda - \rho + k + d p))}{\Gamma_{p, d} (\frac{1}{2} (- \lambda + \rho +
k))} \frac{\Gamma_{p, d} (\frac{1}{2} (- \lambda + \rho + \mu))}{\Gamma_{p, d} (\frac{1}{2} (\lambda + \rho +
\mu))},
\end{equation}
where $c (\mu) = (-1)^{\frac{|\mu| - p k}{2}}$, and $k = 1$ for $\mathbb{K} = \mathbb{R}$ and 
$k = |l|$ for $\mathbb{K} = \mathbb{C}$ (with $l \in \Z$).
\end{theorem}

\begin{proof}
Let $\mu^0 \in \Lambda_\chi^+$ be as in (\ref{eq32}). Then $|\mu^0| = p k$. We identify $\mu^0 = (k, \dotsc, k)$.
The equation (\ref{eq21}) suggests that 
$$\eta_{\mu^0} (\lambda) = (-1)^{p k / 2} G (\lambda)\; \frac{\Gamma_{p, d} (\frac{1}{2} (- \lambda + \rho +
k))}{\Gamma_{p, d} (\frac{1}{2} (\lambda + \rho + k))}. $$
It follows by using (\ref{eq19}) that 
\begin{align}
G (\lambda) & =  (-1)^{\frac{- p k}{2}} \eta_{\mu^0} (\lambda) \frac{\Gamma_{p, d} (\frac{1}{2} (\lambda +
\rho + k))}{\Gamma_{p, d} (\frac{1}{2} (- \lambda + \rho + k))} \nonumber\\
& =  (-1)^{\frac{- p k}{2}} \frac{\Gamma_{p, d} (d (n+1) / 2)}{\Gamma_{p, d} (d p / 2)} \frac{\Gamma_{p,
d} (\frac{1}{2} (\lambda - \rho + k + d p))}{\Gamma_{p, d} (\frac{1}{2} (- \lambda + \rho + k))}. \label{eq29}
\end{align}
Hence, the equation (\ref{eq21}) together with (\ref{eq29}) gives (\ref{eq20}).
\end{proof}

\end{document}